\documentclass[leqno,11pt]{amsart} %
\usepackage{amssymb}
\usepackage{amstext}
\usepackage{graphicx}
\usepackage{tikz}
\theoremstyle{plain} 
\newtheorem{theo}{\indent\sc Theorem}[section]
\newtheorem{lemm}[theo]{\indent\sc Lemma}
\newtheorem{cor}[theo]{\indent\sc Corollary}
\newtheorem{prop}[theo]{\indent\sc Proposition}
\theoremstyle{definition} 
\newtheorem{defi}[theo]{\indent\sc Definition}
\newtheorem{rem}[theo]{\indent\sc Remark}

\newcommand{\vna}{von Neumann algebra}
\newcommand{\masa}{maximal abelian subalgebra}
\newcommand{\Hs}{Hilbert space}

\newcommand{\AOP}{approximative orthogonality property}
\newcommand{\spa}{subfactor planar algebra}

\newcommand{\R}{\mathbb R}
\newcommand{\C}{\mathbb C}
\newcommand{\N}{\mathbb N}
\newcommand{\h}{\mathcal H}
\newcommand{\B}{\mathbb B}

\newcommand{\lN}{{\ell^2(\mathbb N)}}
\newcommand{\lNN}{{\lN\otimes\lN}}
\newcommand{\deltafrac}[1]{\delta^{\frac{#1}{2}}}

\newcommand{\sh}{ S+S^*}
\newcommand{\U}{\frac{\cup-1}{\deltafrac{1}}}
\newcommand{\cupbul}[1]{\cup^{\bullet #1}}
\newcommand{\tv}{\tilde{v}}
\newcommand{\2}{{[-2;2]}}
\newcommand{\Ldeuxnu}{{L^2(\2,\nu)}}
\newcommand{\LMP}{L^2(M)}
\newcommand{\limo}{\lim_{n\rightarrow \omega}}
\newcommand{\LM}{L^2(M)}
\newcommand{\AM}{A\subset M}
\newcommand{\LA}{L^2(A)}
\newcommand{\Pl}{\mathcal P}
\newcommand{\wPl}{{Gr(\Pl)}}

\begin{document}

\title[The cup subalgebra is maximal amenable]{The cup subalgebra of a II$_1$ factor given by a subfactor planar algebra is maximal amenable} 

\author[A. Brothier]{Arnaud Brothier$^*$} 

\thanks{
$^{*}$Supported by ERC Starting Grant VNALG-200749 and by Region Ile de France
}
\address{
Vanderbilt University\endgraf
Department of Mathematics\endgraf
1326 Stevenson Center\endgraf
Nashville\endgraf
TN 37240\endgraf
USA
}
\email{arnaud.brothier@wis.kuleuven.be}
\email{brot@math.jussieu.fr}
\email{arnaud.brothier@gmail.com}


\maketitle

\begin{abstract}
To every subfactor planar algebra was associated a II$_1$ factor with a canonical abelian subalgebra generated by the cup tangle.
Using Popa's approximative orthogonality property, we show that this cup subalgebra is maximal amenable.
\end{abstract}

\section*{Introduction}

The study of \masa s (MASAs) has been initiated by Dixmier \cite{Dixmier_anneaux_max_ab} where he introduced an invariant coming from the normalizer.
Other invariants have been provided later, like the Takesaki equivalence relation \cite{takesaki_invariant_masa}, the length of Tauer \cite{Tauer_masa}, the Pukanszky invariant \cite{pukanszky_invariant} or the $\delta$-invariant \cite{Popa_Singular_masas_in_vna}.

Popa exhibits in \cite{popa_max_inj} an example of a MASA $\AM$ in a II$_1$ factor that is maximal amenable.

This example answers negatively to a question of Kadison that asks if every abelian subalgebra of a II$_1$ factor (with separable predual) is included in a copy of the hyperfinite II$_1$ factor.
We recall that a \vna\ is hyperfinite if and only if it is amenable by the famous theorem of Connes \cite{Connes_76_classification_inj}.
Popa introduced the notion of \AOP\ (in short AOP) and showed that any singular MASA with the AOP is maximal amenable.
Then he proved that the generator MASA in a free group factor is singular and has the AOP.

Using the same scheme of proof, Cameron et al. \cite{Cameron_Fang_Ravichandran_White_10_max_inj} showed that the radial MASA in the free group factor is maximal amenable.
Also Shen \cite{Shen_max_inj_subalg_tensor_prod_free_group_fact},  Jolissaint \cite{Jolissaint_max_inj_and_mixing_masas_group_factors} and Houdayer \cite{Houdayer_12_maxinj} provided other examples of maximal amenable MASAs.

In this paper, we provide maximal amenable MASAs in II$_1$ factors using subfactor planar algebras.
The theory of subfactors has been initiated by Jones \cite{Jones_index_for_subfactors}.
He introduced the standard invariant that has been formalized as a Popa system by Popa \cite{popa_system_construction_subfactor} and as a subfactor planar algebra by Jones \cite{jones_planar_algebra,jones_planar_algebra_II}.
Popa \cite{Popa_Markov_tr_subfactors,popa_system_construction_subfactor,Popa_univ_constr_subfactors} proved that any standard invariant comes from a subfactor. 
Popa and Shlyakhtenko proved \cite{Popa_Shlyakhtenko_univ_prop_LF_subfactor} that the subfactor can be realized in the infinite free group factor $L(\mathbb{F}_\infty)$.
Using planar algebras, random matrix models and free probability, Guionnet et al. \cite{GJS_random_matrices_free_proba_planar_algebra_and_subfactor,JSW_orthogonal_approach_planar_algebra,GJS_semifinite_algebra} showed that any finite depth standard invariant can be realized as a subfactor of an interpolated free group factor. 
Using the same construction, Hartglass \cite{Hartglass_12_GJS} proved that any infinite depth subfactor is realized in $L\mathbb F_\infty$.

The construction of Jones et al. \cite{JSW_orthogonal_approach_planar_algebra}  associated a II$_1$ factor $M$ to a \spa\ $\mathcal P$.
This II$_1$ factor contains a generic MASA $\AM$, see section \ref{subsection_Umaxinj_cup}, that we call the \textit{cup subalgebra}.
The main theorem of this paper is

\begin{theo}\label{main_theo}
For any non trivial \spa\ $\mathcal P$, the cup subalgebra is maximal amenable.
\end{theo}
Note that the construction of \cite{JSW_orthogonal_approach_planar_algebra} has been extended for unshaded planar algebras in \cite{Brot_generalisation_GJSW} and in \cite{Brot_Hartglass_Penneys_12_Rigid_C_tensor_cat}.
In those construction, we have proven that the cup subalgebra is still a MASA. 
It seems very plausible that it is also maximal amenable. 
Note that the cup subalgebra is analogous of the \textit{radial MASA} in a free group factor.
We don't know if for a certain \spa\ those two subalgebras are isomorphic or not.
\section*{Acknowledgments}
I would like to thank the Fondation Sciences Math\'ematiques de Paris which provided me extra support for my stay in UC Berkeley during the spring 2009 when most part of this work was done.
I am happy to thank Melanie MacTavish and Vaughan Jones for making my stay in California very pleasant.

\section{\AOP\ and maximal amenability}\label{section_AOP donne max inj}
We briefly recall Popa's \AOP\ for an abelian subalgebra $\AM$ and how it implies the maximal amenability of $A$, whenever $\AM$ is a singular MASA.
\begin{defi}(see \cite[Lemma 2.1]{popa_max_inj})
Consider a tracial \vna\ $(M,tr)$ and a subalgebra $\AM$.
Let $\omega$ be a free ultrafilter on $\N$.
Then $\AM$ has the \AOP\ (in short AOP) if for any $x\in M^\omega\ominus A^\omega\cap A'$ and any $b\in M\ominus A$ we have $xb\perp bx$ in $L^2(M^\omega)$,
i.e. $\limo tr(x_nbx_n^*b^*)=0$ where $(x_n)_n$ is a representative of $x$.
\end{defi}
\begin{rem}
By polarization, the definition of AOP is equivalent to ask that for any $x_1,x_2\in M^\omega\ominus A^\omega\cap A'$ and any $b_1,b_2\in M\ominus A$ we have $x_1b_1\perp b_2x_2$.
\end{rem}
We recall the fundamental theorem of Popa that is contained in the proof of \cite[Theorem 3.2]{popa_max_inj}.
A more detailed explanation of Popa's theorem has been given in \cite[Lemma 2.2 and Corollary 2.3]{Cameron_Fang_Ravichandran_White_10_max_inj}.
\begin{theo}\label{theo_AOP_donne_maxinj}\cite{popa_max_inj}
Let $\AM$ be a singular MASA with the AOP in a II$_1$ factor $M$.
Then $\AM$ is maximal amenable. 
\end{theo}
\section{Construction of the cup subalgebra}\label{section_constructionGJS_pair}
\subsection{Construction of a II$_1$ factor from a \spa}
Consider a \spa\ $\mathcal P=(\mathcal P_n)_{n\geqslant  0}$ with modulus $\delta>1$.
Let us recall the construction given in \cite{JSW_orthogonal_approach_planar_algebra}.
We assume that the reader is familiar with planar algebras.
For more details on planar algebras, see Jones \cite{jones_planar_algebra,jones_planar_algebra_II} or the introduction of Peters \cite{peters_planar_haagerup_graph}.
Let $Gr(\Pl)$ be the graded vector space equal to the algebraic direct sum $\bigoplus_{n\geqslant 0}\Pl_n$.
We decorate strands in a planar tangle with natural numbers to represent cabling of that strand.
For example 
$$
\begin{tikzpicture}[baseline=.6cm]
\node at (-.2,.55){$k$};
 \draw (0,-.42)--(0,1.4);
 \node at (.5,.5){$=$};
 \end{tikzpicture}
 \begin{array}{c}
\scriptsize{k}\\
 \overbrace{
 \begin{tikzpicture}[baseline=.6cm]
 \draw (1,0)--(1,1.5);
  \draw (1.5,0)--(1.5,1.5); 
  \node at (1.25,.75){$\cdot$};
\end{tikzpicture}
}
\end{array}
$$
An element $a\in\Pl_n$ will be represent as a box:
$$
a=
\begin{tikzpicture}[baseline=.6cm]
\draw(0,1.5)--(2,1.5)--(2,0)--(0,0)--(0,1.5);
\draw (.6,.3)--(.6,1.1)--(1.4,1.1)--(1.4,.3)--(.6,.3);
\draw (1,1.1)--(1,1.5);
\node at (1.2,1.3){\scriptsize$2n$};
\node at (1,.68){$a$};
\end{tikzpicture}\ .
$$
We assume that the distinguished first interval is at the top left of the box.
We consider the inner product $\langle \cdot,\cdot\rangle$ on each $\mathcal P_n$ that is:
$$
\langle a,b\rangle=
\begin{tikzpicture}[baseline=.6cm]
\draw(0,1.4)--(2.8,1.4);
\draw(0,0)--(2.8,0);
\draw(0,0)--(0,1.4);
\draw(2.8,0)--(2.8,1.4);
\draw (1.1,.7)--(1.7,.7);	
	\draw(.3,1.1)--(1.1,1.1)--(1.1,.3)--(.3,.3)--(.3,1.1);
	\draw(1.7,1.1)--(2.5,1.1)--(2.5,.3)--(1.7,.3)--(1.7,1.1);
	\node at (.7,.7) {$a$};
	\node at (2.1,.7) {$b^*$};
         \node at (1.4,.9){\scriptsize$2n$};
\end{tikzpicture}\ ,\ \text{for all}\ a,b\in\mathcal P_n.
$$
We extend this inner product on $Gr(\Pl)$ in such a way that the spaces $\Pl_n$ are pairwise orthogonal.
We still write $\Pl_n$ when it is considered as the $n$-graded part of $Gr(\Pl)$.
Let $\h$ be the Hilbert space equal to the completion of $Gr(\Pl)$ for its prehilbert structure.
Note that $\h$ is the \Hs\ equal to the orthogonal direct sum of the spaces $\Pl_n$.
We define a multiplication on $\wPl$ given by the tangle:
$$
ab=
\sum^{\min(2n,2m)}_{
\substack{
j=0
}}
\begin{tikzpicture}[baseline = .6cm]
 \draw (-1.25,1.7)--(2.5,1.7);
  \draw (-1.25,-.2)--(2.5,-.2);
  \draw (-1.25,-.2)--(-1.25,1.7);
    \draw (2.5,-.2)--(2.5,1.7);
\draw (-.2,.8)--(-.2,1.7);	
	\draw (1.4,.8)--(1.4,1.7);
	\draw (.2,.8) arc (180:0:.4);	
\draw (-.4,.8)--(.4,.8)--(.4,-0)--(-.4,-0)--(-.4,.8);
	\draw(.8,.8)--(1.6,.8)--(1.6,-0)--(.8,-0)--(.8,.8);
	\node at (0,.4) {$a$};
	\node at (1.2,.4) {$b$};
	\node at (-.65,1.2) {{\scriptsize{$2n-j$}}};
	\node at (1.90,1.2) {{\scriptsize{$2m-j$}}};
	\node at (.6,1.4) {{\scriptsize{$j$}}};
\end{tikzpicture}\ ,\ \text{for all}\ a\in\mathcal P_n,\ b\in\mathcal P_m.
$$
For a fix $a\in Gr(\Pl)$, the map $b\in \wPl\longmapsto ab\in\wPl$ is bounded for the inner product $\langle\cdot,\cdot\rangle$.
This gives us a representation of the $*$-algebra $\wPl$ on $\h$.
We denote by $M$ the \vna\, equal to the bicommutant of this representation.
It is a II$_1$ factor by \cite{JSW_orthogonal_approach_planar_algebra}.
We identify the graded algebra $\wPl$ and its image in the \vna\ $M$.
The unique faithful normal trace $tr$ of $M$ is the one coming from the planar algebra structure of $\Pl$.
It is equal to the formula $tr(a)=\langle a,1\rangle$, where $1$ is the unity of $Gr(\Pl)$.
Let $\LMP$ be the \Hs\ coming from the GNS construction over the trace $tr$.
Note that the standard representation of the \vna\ $M$ on the \Hs\ $\LMP$ is conjugate to the action of $M$ on the \Hs\ $\h$.
We will identify those two representations.
Also, we identify $M$ with its image in $\LMP$.
The left and right action of $M$ on the Hilbert space $L^2(M)$ are denote by $\pi$ and $\rho$,
i.e. $\pi(x)\rho(y)z=xzy,$ for $x,y,z\in M$.
The norm of $M$ (resp. $\LM$) is denoted by $\Vert\cdot\Vert$ (resp. $\Vert\cdot\Vert_2$). 
It the context is sufficiently clear, we would denote the norm of $\LM$ by $\Vert\cdot\Vert$.
We define a multiplication on $\wPl$ by requiring that if $a\in\Pl_n$ and $b\in\Pl_m$, then $a\bullet b\in\Pl_{n+m}$ is given by
$$
a\bullet b=
\begin{tikzpicture}[baseline=.6cm]
\draw (0,1.5)--(2.8,1.5)--(2.8,0)--(0,0)--(0,1.5);
\draw(.3,1.1)--(1.1,1.1)--(1.1,.3)--(.3,.3)--(.3,1.1);
	\draw(1.7,1.1)--(2.5,1.1)--(2.5,.3)--(1.7,.3)--(1.7,1.1);
	\node at (.7,.7) {$a$};
	\node at (2.1,.7) {$b$};
\draw (.7,1.1)--(.7,1.5);
\draw (2.1,1.1)--(2.1,1.5);
\node at (.9,1.3){\scriptsize$2n$};
\node at (2.35,1.3){\scriptsize$2m$};
\end{tikzpicture}\ .
$$
We remark that $\Vert a\bullet b\Vert_2=\Vert a\Vert_2\Vert b\Vert_2$, for all $a\in \Pl_n$ and $b\in\Pl_m$.
By the triangle inequality, the bilinear function
\begin{align*}
Gr(\Pl)\times Gr(\Pl)&\longrightarrow Gr(\Pl)\\
(a,b)&\longmapsto a\bullet b
\end{align*}
is continuous for the norm $\Vert\cdot\Vert_2$.
We extend this operation on $L^2(M)\times L^2(M)$ and still denote it by $\bullet$.
\subsection{The cup subalgebra}\label{subsection_Umaxinj_cup}
The cup subalgebra $A\subset M$ is the abelian \vna\, generated by the self-adjoint element cup:
$$
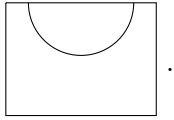
\begin{tikzpicture}[baseline=.6cm]
\draw(0,1.5)--(2,1.5)--(2,0)--(0,0)--(0,1.5);
\draw (.3,1.5) arc (-180:0:.7);
\end{tikzpicture}\ .
$$
We denote cup by the symbol $\cup$.
Also we use the following notation
$$
\cup^{\bullet k}=
\begin{array}{c}
k\ \text{cups}\\
\overbrace{
\begin{tikzpicture}[baseline=.6cm]
\draw(0,1.5)--(3,1.5);
\draw(0,0)--(3,0);
\draw(0,0)--(0,1.5);
\draw(3,0)--(3,1.5);
\draw (.2,1.5) arc (-180:0:.4);
\draw(2,1.5) arc (-180:0:.4);
\node at (1.5,1.3) {$\cdots$};
\end{tikzpicture}
}
\end{array}\ .
$$
We use the convention that $0=\cupbul{k}$ for $k<0$ and $1=\cupbul{0}$. 
Let $n\geqslant 1$ and $V_n$ be the subspace of $\Pl_n$ of elements which vanish when a cap is placed at the top right and vanish when a cap is placed at the top left, i.e.
$$V_n=\left\{
a\in \mathcal P_n,\
\begin{tikzpicture}[baseline=.6cm]
\draw (0,0)--(0,1.5)--(2.4,1.5)--(2.4,0)--(0,0);
\draw (.2,.2)--(.2,1)--(1.8,1)--(1.8,.2)--(.2,.2);
\draw (1.4,1)--(1.4,1.5);
\draw (.4,1) arc (180:0:.4);	
\node at (1.9, 1.2){\scriptsize{$2n-2$}};
\node at (1,.6){$a$};
\node at (2.7,.6){$=$};
\draw (3,0)--(3,1.5)--(5.4,1.5)--(5.4,0)--(3,0);
\draw (3.6,.2)--(3.6,1)--(5.2,1)--(5.2,.2)--(3.6,.2);
\draw (4.2,1) arc (180:0:.4);
\draw (4,1)--(4,1.5);
\node at (3.5, 1.2){\scriptsize{$2n-2$}};
\node at (4.4,.6){$a$};
\end{tikzpicture}=0
\right\}\ .
$$
We denote by $V$ the orthogonal direct sum of the $V_n$, i.e. $$V=\bigoplus_{n=1}^\infty V_n.$$
Let $\lN$ be the separable Hilbert space with orthonormal basis $\{e_n,\ n\geqslant 0\}$ and $S\in\B(\lN)$ the unilateral shift operator.
\begin{prop}\label{prop_Umaxinj_bimodule_structure_LM}
The map
$$\begin{array}{cccccc}
 \Theta: & L^2(M) & \longrightarrow & \lN & \oplus & (\lN\otimes V\otimes \lN) \\
   & \delta^{-\frac{k}{2}} \cupbul k & \longmapsto & e_k&\oplus  & 0 \\
   & \delta^{-\frac{l+r}{2}} \cupbul l\bullet v\bullet \cupbul r & \longmapsto & 0 & \oplus &  e_l\otimes v\otimes e_r
\end{array}$$ defines a unitary transformation, where $k,l,r\geqslant 0$, $v\in V$ and $\delta$ is the modulus of the planar algebra.
We have that
$$\Theta\pi(\U)\Theta^*=\left( \begin{array}{cc}
\sh-q_{e_0}&0\\
0&(S+S^*)\otimes 1_{V}\otimes 1_\lN
\end{array}\right)$$
and
$$\Theta\rho(\U)\Theta^*=\left( \begin{array}{cc}
\sh-q_{e_0}&0\\
0&1_\lN\otimes 1_{V}\otimes (\sh)
\end{array}\right),$$
where $q_{e_0}$ is the rank one projection on $\C e_0$ and $1_V,1_\lN$ are the identity operators of the Hilbert spaces $V$ and $\lN$.
\end{prop}
\begin{proof}
See \cite{JSW_orthogonal_approach_planar_algebra}[theorem 4.9.].
\end{proof}
\begin{cor}\label{cor_Umaxinj_cup_sing_MASA}
The cup subalgebra is a singular MASA.
\end{cor}
\begin{proof}
The $A$-bimodule $\LM\ominus\LA$ is isomorphic to an infinite direct sum of the coarse bimodule $\LA\otimes\LA$.
This implies that $\AM$ is maximal abelian. 
See \cite{JSW_orthogonal_approach_planar_algebra} for more details.
Suppose that there exists a unitary $u$ in the normalizer of $A$ inside $M$ which is orthogonal to $A$. 
It generates a $A$-subbimodule 
\begin{equation}\label{equa:cor_singular}
\mathcal K\subset \bigoplus_{j=0}^\infty \LA\otimes \LA.
\end{equation}
We have the inclusion \ref{equa:cor_singular} if and only if the automorphism $a\in A\mapsto uau^*$ is trivial.
This implies that $u\in A'\cap M$. Hence $u\in A$, a contradiction.
Therefore, $\AM$ is singular.
\end{proof}
\subsection{Basic facts on the unilateral shift operator}\label{section_shift}
Consider the semi-circular measure 
$$d\nu(t)=\frac{\sqrt{4-t^2}}{2\pi}dt$$ defined on the interval $\2$.
Let $P_i\in\R[X]$ be the family of polynomials such that $P_0(X)=1,\ P_1(X)=X$ and $P_i(X)=XP_{i-1}(X)-P_{i-2}(X)$ for all $i\geqslant 2$.
By \cite[example 3.4.2]{Voiculescu_dykema_nica_Free_random_variables}, we have that the map
\begin{align*}
\Psi:\lN&\longrightarrow\Ldeuxnu\\
e_i&\longmapsto P_i
\end{align*} 
defines a unitary transformation.
Furthermore, for any continuous function $f\in\mathcal C(\2)$ we have that $(\Psi^*f(\sh)\Psi)(t)=tf(t)$, for almost every $t\in\2$.

\begin{lemm}\label{lem_sum_Pi_div}
Consider $I\geqslant 0$ and the function $R_I:\2\longrightarrow \R$ such that $R_I(t)=\sum_{i=0}^I P_i(t)^2$.
The sequence $(R_I)_{I\geqslant 0}$ converges uniformly to $+\infty$.
\end{lemm}

\begin{proof}
Let us prove the simple convergence to $+\infty$.
Suppose there exists $t_0\in\2$ such that the sequence $(R_I(t_0))_k$ does not converge to $+\infty$.
The polynomiasl $P_i$ have real coefficient. 
Hence, for any $t\in\2$, $P_i(t)$ is real; thus, $(R_I(t_0))_k$ is an increasing sequence in $\R$.
If this sequence does not diverge, then it is bounded.
Then, the sequence $(P_i(t_0))_i$ is square summable.
In particular we have $\lim_{i\rightarrow \infty} P_i(t_0)=0.$
We put $\varepsilon_i=P_i(t_0)$.
We have that $\varepsilon_{i+1}=t_0\varepsilon_i-\varepsilon_{i-1}$ and $\lim_{i\rightarrow \infty}\varepsilon_i=0.$
There is only one sequence that satisfies those axioms and it is the sequence equal to zero.
Since $0\neq 1=P_0(t_0)=\varepsilon_0$, we arrive to a contradiction and thus, $\lim_{I\rightarrow \infty}S_I(t)=+\infty$ for any $t\in\2$.
To conclude we use the following well known result due to Dini:
Let $(f_I)_I$ be a sequence of continuous functions from a compact topological space $K$ to $\R$ such that $f_I\leqslant f_{I+1}$.
If for any $t\in K$, $\lim_{I\rightarrow \infty} f_I(t)=+\infty,$ then the sequence $(f_I)_I$ converges uniformly to $+\infty$.
\end{proof}
\subsection{Proof of Theorem \ref{main_theo}}
According to the Theorem \ref{theo_AOP_donne_maxinj} and Corollary \ref{cor_Umaxinj_cup_sing_MASA} it is sufficient to show that the cup subalgebra has the AOP.
Fix $x\in M^\omega\ominus A^\omega\cap A'$ and $b\in M\ominus A$.
Let us show that $xb\perp bx$.
By the Kaplansky density theorem we can assume that their exists $J\geqslant 1$ such that $b\in\bigoplus_{j=0}^J\mathcal P_j$.
Suppose that $\Vert x\Vert\leqslant 1$ and fix a sequence $x_n\in M$ which is a representative of $x$ such that $x_n\in M\ominus A$ and $\Vert x_n\Vert\leqslant 1$ for all $n\geqslant 0$.

Consider the closed subspaces of $\LM$:
\begin{align*}
       Y_L&=\overline{\text{span}}\{\cupbul l\bullet v\bullet \cupbul r,\ l,r\leqslant L,\ v\in V\}\ \text{and}\\
    Z_L&=\overline{\text{span}}\{ \cupbul l \bullet v \bullet \cupbul r,\ l\text{ or }r\leqslant L,\,v\in V\},
\end{align*}
for all $L\geqslant 0$.
Remark that $b$ is in $Y_{J-1}$.

We claim that for any $z\in M$ which is orthogonal to $A$ and $Z_{J-1}$ we have
 \begin{equation}\label{equa_zaperpby}
 zb\perp bz.
 \end{equation}
The element $z$ is a weak limit of finite linear combinations of $\cupbul i\bullet v\bullet \cupbul j,$ where $i,j\geqslant J$ and $v\in V$.
The element $b$ is a finite linear combination of $\cupbul k\bullet \tilde v\bullet \cupbul r,$ where $k,r\leqslant J-1$ and $\tilde v\in V$.
We have
\begin{align*}
(\cupbul i\bullet v\bullet \cupbul j)(\cupbul k\bullet \tv\bullet \cupbul r) = &(\cupbul i\bullet v\bullet \cupbul{j+k}\bullet \tv\bullet\cupbul r)+(\cupbul i\bullet v\bullet \cupbul{j+k-1}\bullet\tv\bullet\cupbul r)+\cdots \\
 & +\delta^k (\cupbul i\bullet v\bullet \cupbul{j-k}\bullet\tv\bullet \cupbul r)+\delta^k (\cupbul i\bullet v\bullet \cupbul{j-k-1}\bullet\tv\bullet\cupbul r),
\end{align*}
for any $i,j\geqslant J$ and $k,r\leqslant J-1$.
It is easy to see that $v\bullet\cupbul n\bullet\tilde v$ is an element of $V$ for any $n$.
Hence, the product  $(\cupbul i\bullet v\bullet \cupbul j)(\cupbul k\bullet \tv\bullet \cupbul r)$ is in the vector space
$$\overline{\text{span}}\{\cupbul l\bullet w\bullet \cupbul r, \ l\geqslant J,\ w\in V,\ r\leqslant J-1\}$$
and so does $zb$.
A similar computation shows that $bz$ is in the closed vector space
$$\overline{\text{span}}\{\cupbul l\bullet v\bullet \cupbul r,\  l\leqslant J-1,\ w\in V,\ r\geqslant J \}.$$
Therefore, we have $zb\perp bz$. 
This proves \ref{equa_zaperpby}.
Hence, if we show that $x$ is in the orthogonal of $Z_{J-1}^\omega$ then we would have proven that $xb$ is orthogonal to $bx$.
Consider $Q_J:\LM\longrightarrow Z_{J-1}$, the orthogonal projection of range $Z_{J-1}$. 
We remark that 
$$\Theta Q_J\Theta^*=\bigoplus_{j=0}^{J-1}((q_{e_j}\otimes 1_V\otimes 1_\lN)\oplus (1_\lN\otimes1_V\otimes q_{e_j})),$$
where $\Theta$ is the unitary transformation defined in Section \ref{section_shift} and $1_V,1_\lN$ are the identity operators of $V$ and $\lN$.
By symmetry, it is sufficient to show that 
\begin{equation}\label{equa_qitime1time1}
\limo \Vert(q_{e_j}\otimes 1_V\otimes 1_\lN) \xi_n\Vert=0,\ \text{for any}\ j\geqslant 0,
\end{equation}
where $\xi_n:=\Theta(x_n)$.
We know that $x\in M^\omega\cap A'$. 
Hence by conjugation by $\Theta$ we obtain the equation
\begin{equation}\label{equa_shift_xi_n}
 \limo\Vert ((\sh)\otimes 1_V\otimes 1_\lN -1_\lN\otimes 1_V\otimes(\sh))\xi_n\Vert=0.
 \end{equation}
We will show that \ref{equa_shift_xi_n} implies \ref{equa_qitime1time1}.

All the operators involved in our context act trivially on the factor $V$. 
For simplicity of the notations we stop writing the extra "$\otimes 1_V\otimes$" in the formula and denote the identity operator $1_\lN$ by $1$.
Therefore, we assume that $\xi_n$ is a vector of $\lNN$. 
The equation \ref{equa_qitime1time1} and \ref{equa_shift_xi_n} becomes
\begin{equation}\label{equa_qitime1time1deux}
\limo \Vert(q_{e_i}\otimes 1) \xi_n\Vert=0,\ \text{for any}\ i\geqslant 0\ \text{and}
\end{equation}
\begin{equation}\label{equa_shift_xi_ndeux}
 \limo\Vert ((\sh)\otimes 1-1\otimes(\sh))\xi_n\Vert=0.
 \end{equation}
Consider the partial isometry $v_i\in\B(\lN)$ such that $v_i^*v_i=q_{e_i}$ and $v_iv_i^*=q_{e_0}$.
We claim that for all $i\geqslant 0$ we have 
\begin{equation}\label{equa_qe0Pi-qei}
\limo \Vert ((v_i\otimes 1)-(q_{e_0}\otimes P_i(\sh)))\xi_n\Vert=0,
\end{equation}
where $\{P_i\}_i$ is the family of polynomials defined in Section \ref{section_shift}.
Remark that for all $k\geqslant 2$ we have
\begin{equation*}
(\sh)^k\otimes 1-1\otimes (\sh)^k=((\sh)\otimes 1-1\otimes (\sh))\circ( \sum_{j=0}^{k-1}(\sh)^j\otimes (\sh)^{k-1-j}).
\end{equation*}
Therefore, the equation \ref{equa_shift_xi_ndeux} implies that
$$\limo\Vert(P(\sh)\otimes 1-1\otimes P(\sh))\xi_n\Vert=0,\ \text{for all polynomials}\ P.$$
In particular, 
$$\limo\Vert(P_i(\sh)\otimes 1-1\otimes P_i(\sh))\xi_n\Vert=0,\ \text{for all}\ i\geqslant 0.$$
Note that
$$P_i(\sh)(e_0)=e_i,\ \text{for all}\ i\geqslant 0.$$
Furthermore, $P_i$ has real coefficient. 
Therefore, the operator $P_i(\sh)$ is self-adjoint.
We have
\begin{align*}
\langle q_{e_0}\circ P_i(\sh)e_l,e_r\rangle&=\langle P_i(\sh)e_l,q_{e_0}e_r\rangle=\delta_{r,0}\langle P_i(\sh)e_l,e_0\rangle\\
&=\delta_{r,0}\langle e_l,P_i(\sh)e_0\rangle=\delta_{r,0}\delta_{l,i},
\end{align*}
where $i,l,r\geqslant 0$ and $\delta_{n,m}$ is the Kronecker symbol.
This shows that $q_{e_0}\circ P_i(\sh)=v_i,$ for all $i\geqslant 0$.
We have
$$\limo\Vert(q_{e_0}\otimes 1)\circ(P_i(\sh)\otimes 1-1\otimes P_i(\sh))\xi_n\Vert=0.$$
Therefore, we have $$\limo\Vert(v_i\otimes 1-q_{e_0}\otimes P_i(\sh))\xi_n\Vert=0.$$
This proves the claim. We have
$$\limo\Vert(q_{e_i}\otimes 1-v_i^*q_{e_0}\otimes P_i(S+S^*))\xi_n\Vert=0.$$
This means that 
$$\limo\Vert(q_{e_i}\otimes 1)\xi_n-(v_i^*\otimes P_i(\sh))\circ (q_{e_0}\otimes 1)\xi_n\Vert=0.$$
Hence, we have
\begin{align*}
\limo\Vert (q_{e_i}\otimes 1) \xi_n\Vert&\leqslant \limo\Vert (v_i^*\otimes P_i(\sh))\circ (q_{e_0}\otimes 1)\xi_n\Vert\\
& \leqslant \Vert v_i^*\otimes P_i(\sh)\Vert \limo\Vert (q_{e_0}\otimes 1) \xi_n\Vert.
\end{align*}
Therefore, to prove \ref{equa_qitime1time1deux} it is sufficient to show that
\begin{equation*}
\limo \Vert(q_{e_0}\otimes 1)\xi_n\Vert=0.
\end{equation*}

Let us fix $\varepsilon>0$, we have to find an element of the ultrafilter $E\in\omega$ such that for any $n\in E$,
$\Vert (q_{e_0}\otimes 1)\xi_n\Vert<\varepsilon.$
By the triangle inequality, we have
$$\Vert (q_{e_0}\otimes P_i(\sh))\xi_n\Vert\leqslant \Vert (q_{e_0}\otimes P_i(\sh))\xi_n-(v_i\otimes 1)\xi_n\Vert+\Vert (v_i\otimes 1)\xi_n\Vert,$$
for all $i\geqslant 0$.
We have $\Vert (v_i\otimes 1)\xi_n\Vert \leqslant \Vert\xi_n\Vert \leqslant 1$; thus,
\begin{align}\label{equa_align_qeixin}
\Vert (v_i\otimes 1)\xi_n\Vert^2\geqslant &\Vert (q_{e_0}\otimes P_i(\sh))\xi_n\Vert^2 \\
&-\Vert (q_{e_0}\otimes P_i(\sh))\xi_n-(v_i\otimes 1)\xi_n\Vert^2\notag \\
&-2\Vert (q_{e_0}\otimes P_i(\sh))\xi_n-(v_i\otimes 1)\xi_n\Vert.\notag
\end{align}
By Lemma \ref{lem_sum_Pi_div}, there exists an integer $I\in\N$ such that
$$\inf_{t\in\2}S_I(t)>\frac{2}{\varepsilon}.$$

We have
\begin{align}\label{equa_sumqeoPi}
\sum_{i=0}^I\Vert (q_{e_0}\otimes P_i(\sh))\xi_n\Vert^2&=\sum_{i=0}^I\Vert (1\otimes P_i(\sh))\circ(q_{e_0}\otimes 1)\xi_n\Vert^2\notag \\
&=\sum_{i=0}^I\int_\2\Vert P_i(t)((q_{e_0}\otimes \Psi)\xi_n)(t)\Vert^2d\nu(t)\notag\\
&=\int_\2(\sum_{i=0}^IP_i(t)^2)\Vert ((q_{e_0}\otimes \Psi)\xi_n)(t)\Vert^2d\nu(t)\notag\\
&\geqslant\frac{2}{\varepsilon}\Vert (q_{e_0}\otimes \Psi)\xi_n\Vert^2=\frac{2}{\varepsilon}\Vert (q_{e_0}\otimes 1)\xi_n\Vert^2,
\end{align}
where $\Psi$ is the unitary transformation defined in section \ref{section_shift}.

By \ref{equa_qe0Pi-qei}, there exists an element of the ultrafilter $E\in\omega$ such that for any $n\in E$ and $i\in\{0,\cdots,I\}$ we have
\begin{equation}\label{equa_Umaxinj_qeo_J}
    \Vert ((q_{e_0}\otimes P_i(\sh))-(v_i\otimes 1))\xi_n\Vert<\frac{1}{4}.
\end{equation}
By Pythagoras theorem and the inequalities \ref{equa_align_qeixin}, \ref{equa_sumqeoPi} and \ref{equa_Umaxinj_qeo_J} we have
\begin{align*}
1&\geqslant\Vert \xi_n\Vert^2=\sum_{i\geqslant 0}\Vert (q_{e_i}\otimes 1)\xi_n\Vert^2\\
&\geqslant \sum_{i=0}^I\Vert (q_{e_i}\otimes 1)\xi_n\Vert^2=\sum_{i=0}^I\Vert (v_i\otimes 1)\xi_n\Vert^2 \\
&\geqslant \sum_{i=0}^I \Vert (q_{e_0}\otimes P_i(\sh))\xi_n\Vert^2-(I+1) (\frac {1} { 4^2}+ 2 \frac {1} {4})\\
&\geqslant \frac{2(I+1)}{\varepsilon}\Vert (q_{e_0}\otimes 1)\xi_n\Vert-(I+1).
\end{align*}
This implies
$$\Vert (q_{e_0}\otimes 1)\xi_n\Vert\leqslant \varepsilon,\ \text{for all}\ n\in E.$$
We have proved that
$$\limo \Vert(q_{e_0}\otimes 1)\xi_n\Vert_2=0.$$
Therefore, $\limo \Vert Q_J(x_n)\Vert=0$ which implies that $x$ is orthogonal to $Z_{J-1}^\omega$. 
The equality \ref{equa_zaperpby} implies that $xb\perp bx$.
Thus, the cup subalgebra $A\subset M$ has the AOP.
By Corollary \ref{cor_Umaxinj_cup_sing_MASA}, $\AM$ is a singular MASA.
Hence, by Theorem \ref{theo_AOP_donne_maxinj}, the cup subalgebra is maximal amenable.

\bibliographystyle{alpha}

\end{document}